\documentclass{article}
\usepackage[utf8]{inputenc}
\usepackage{amsthm}
\usepackage{amsmath}
\usepackage{amsfonts}
\usepackage{color}
\usepackage{latexsym}
\usepackage{geometry}
\usepackage{enumerate}
\usepackage[shortlabels]{enumitem}
\usepackage{csquotes}
\usepackage{graphicx}
\usepackage{comment}
\usepackage{appendix}
\usepackage{hyperref}
\hypersetup{
    colorlinks=true,
    linkcolor=blue,
    filecolor=magenta,      
    urlcolor=cyan,
}
\usepackage{bm}
\usepackage{amssymb}

\emergencystretch=\maxdimen
\def \NN {\mathcal{N}}

\def \dist {{\rm dist}}
\def \IR {\mathbb{R}}

\DeclareMathOperator{\Rm}{Rm}
\DeclareMathOperator{\Ric}{Ric}

\makeatletter
\newcommand*{\rom}[1]{\rm {\expandafter\@slowromancap\romannumeral #1@}}
\makeatother

\def\XXint#1#2#3{{\setbox0=\hbox{$#1{#2#3}{\int}$ }
\vcenter{\hbox{$#2#3$ }}\kern-.6\wd0}}

\protected\def\vts{%
  \ifmmode
    \mskip0.5\thinmuskip
  \else
    \ifhmode
      \kern0.08334em
    \fi
  \fi
}

\numberwithin{equation}{section}
\newtheorem{Theorem}{Theorem}[section]
\newtheorem{Proposition}[Theorem]{Proposition}
\newtheorem{Lemma}[Theorem]{Lemma}
\newtheorem{Corollary}[Theorem]{Corollary}
\theoremstyle{definition}
\newtheorem{Definition}[Theorem]{Definition}

\def \NN {\mathcal{N}}

\newcommand\blfootnote[1]{%
  \begingroup
  \renewcommand\thefootnote{}\footnote{#1}%
  \addtocounter{footnote}{-1}%
  \endgroup
}

\title{Lower bounds for the scalar curvatures of Ricci flow singularity models}
\author{Pak-Yeung Chan$^*$,\blfootnote{Department of Mathematics, University of California San Diego, La Jolla, CA 92093} Bennett Chow$^*$, Zilu Ma$^{**}$,\blfootnote{$^*$ Department of Mathematics, Rutgers University, Piscataway, NJ 08854} Yongjia Zhang$^{***}$\blfootnote{$^{**}$ School of Mathematical Sciences, Shanghai Jiao Tong University, Minhang District, Shanghai, China 200240}}
\date{}

\begin{document}


\maketitle
\begin{abstract}
    In a series of papers, Bamler \cite{Bam20a,Bam20b,Bam20c} further developed the high-dimensional theory of Hamilton's Ricci flow to include new monotonicity formulas, a completely general compactness theorem, and a long-sought partial regularity theory analogous to Cheeger--Colding theory.
In this paper we give an application of his theory to lower bounds for the scalar curvatures of singularity models for Ricci flow.
In the case of $4$-dimensional non-Ricci-flat steady soliton singularity models, we obtain as a consequence a quadratic decay lower bound for the scalar curvature.
\end{abstract}

\bigskip 

\section{Introduction and background}

To help formulate Ricci flow with surgery, one would like to better understand singularity models of the Ricci flow. This has been highly successful in dimension $3$ by the Hamilton--Perelman theory leading to the proof of the Poincar\'e and Thurston geometrization conjectures by Hamilton's
Ricci flow \cite{Per02,Per03a,Per03b}.
A deeper understanding of Ricci flow and singularity formation in dimension 3 has led to solutions to a pair of fundamental conjectures of Perelman; see Brendle \cite{Bre13,Bre20} and Bamler \cite{Bam18} and the references therein. Bamler and Kleiner \cite{BK22,BK21a} proved the generalized Smale conjecture using $3$-dimensional Ricci flow with surgery.
Furthermore, the understanding of $3$-dimensional ancient solutions essentially complete. See Brendle, Daskalapoulos, and Sesum \cite{BDS21}, Angenent, Brendle, Daskalapoulos, and Sesum \cite{ABDS22},
Bamler and Kleiner \cite{BK21b}, and Lai \cite{Lai20,Lai22}.

Among the higher dimensions, dimension 4 is the most hopeful. In this case, 
Bamler's theory \cite{Bam20a,Bam20b,Bam20c} yields the strongest results regarding Ricci flow singularity formation and compactness.
In general, regarding singularity models, one is interested in Ricci solitons since they are prototypical \cite{Ham88,Ham93} and \cite{Per02,Per03a}.
Among a vast literature on Ricci solitons,
important progress has been made by Munteanu and Wang (see \cite{MW19} and the references therein), especially on the geometric understanding of shrinkers.

A general question is: what curvature estimates hold for singularity models, and in particular, Ricci solitons? In this paper we apply Bamler's theory to make some further progress on lower bounds for the scalar curvature.

If a singularity model is Ricci flat, then by Perelman's no local collapsing theorem \cite{Per02}, it has Euclidean volume growth. If, in addition, the dimension is $4$, then by a theorem of Cheeger and Naber \cite[Corollary 8.86]{CN15}, it is an asymptotically locally Euclidean (ALE) space.
The Eguchi--Hanson metric, which is such a space, occurs as a singularity model of $4$-dimensional Ricci flow by Appleton \cite{App19}.
In this paper we consider the complementary case of non-Ricci-flat singularity models.

Let $(M^n,g_t)_{t\le 0}$ be an ancient Ricci flow, i.e., $\partial_t g_t = - 2 \operatorname{Ric}_{g_t}$, and $t\leq 0$ means that $t\in (-\infty,0]$, 
whose time slices are complete.
The conjugate heat operator is defined as $\Box^* = - \partial_t - \Delta_{g_t} + R_{g_t}$.
We write the conjugate heat kernel based at
the space-time point
$(o,0)$ as
\[
    K(o,0\,|\, x,t) = (4\pi|t|)^{-\frac{n}{2}} e^{-f_t(x)}.
\]
So, $\Box^*_{x,t} K=0$ and $\lim_{t\to 0_-}K(o,0\,|\,\cdot,t) = \delta_o$.
The pointed Nash entropy based at $(o,0)$ is
\begin{equation}
    \NN_{o,0}(\tau) : = \int_M f_t(x) \, K(o,0\,|\, x,t) \, dg_t (x) - \frac{n}{2},
\end{equation}
where $\tau=|t|$ (we will henceforth use this notation).

Henceforth, we assume that $(M^n,g_t)_{t\le 0}$ either is a finite-time singularity model or has bounded curvature over compact time intervals with a pointed \emph{Nash entropy lower bound}, namely,
\begin{equation}\tag{NELB}
    \label{eq: Nash entropy lower bd}
        \inf_{\tau>0} \NN_{o,0}(\tau) =: \mu_\infty > -\infty,
\end{equation}
for some point $o\in M$ (and thus for any point since $\inf\NN$ is independent of the base point by \cite[Proposition 4.6]{MZ21}). 
By a \emph{finite-time singularity model}
we mean a complete and non-flat ancient solution $(M^n,g_t)_{t\le 0}$ which arises as a blow-up limit of a Ricci flow $(\overline{M}^n,\overline{g}(t))_{t\in [0,T)}$ on a closed manifold $\overline{M}$, with a finite-time singularity at $T<\infty$. More precisely, one can find sequences of $(x_i,t_i)\in \overline{M}\times [0,T)$ and $\lambda_i \to\infty$ such that 
the rescaled solutions
$\big( \overline{M},\lambda_i\overline{g}(t_i+\tfrac{t}{\lambda_i}) \big) , t\in (-\lambda_it_i,0]$, $\mathbb{F}$-converge on compact time intervals to $(M, g(t)), t\in (-\infty, 0]$, in the sense of Bamler \cite{Bam20b} (see also \cite{CFSZ20, BCDMZ21}).

For each $t<0,$ let $(z_t,t)$ be an $H_n$-center of $(o,0)$, as defined in \cite{Bam20a}. That is, $\operatorname{Var}_t\big(\delta_{z_t},K(o,0\,|\, \cdot,t) \vts dg_t \big)  \leq H_n|t|$, where $H_n:= \frac{(n-1)\vts\pi^2}{2}+4$.
Such space-time points effectively represent the past of a point $(o,0)$ of a Ricci flow. 
We shall consider the following 
$H_n$-center
\emph{scalar curvature lower bound}
assumption 
\begin{equation}
    \tag{SCLB}
    \label{eq: sc lower bd near H-centers}
    \inf\Big\{(1+\tau)\,R(x,t)\,\Big|\,x\in B_t(z_t,D\sqrt{1+\tau}), t<0\Big\}:=a_0>0,
\end{equation}
for some constant $D<\infty.$
As we will see below, this assumption is suitable for applications to 
noncollapsed
steady solitons.


\section{Statements of the results}

The main results of this paper stem from the following global scalar curvature 
decay estimate for the ancient solutions in which we are interested.

\begin{Theorem}
\label{thm: general thm}
    Assume that a complete ancient solution $(M^n,g_t)$ is a finite-time singularity model or that $(M,g_t)$ has bounded curvature over compact time intervals with \eqref{eq: Nash entropy lower bd}.
    Further assume \eqref{eq: sc lower bd near H-centers}.
    If the constant $D$ in \eqref{eq: sc lower bd near H-centers} satisfies $D\ge \underline{D}(n,\mu_\infty)$, then
    \[
       (1+\tau)\, R(x,t) \ge \frac{c(n,\mu_\infty,a_0)}{f_t(x) + C(n) - \mu_\infty},
    \]
    for $(x,t)\in M\times(-\infty,0)$, where $a_0$ is the infimum in \eqref{eq: sc lower bd near H-centers} and $\mu_\infty$ is as in \eqref{eq: Nash entropy lower bd}.
\end{Theorem}


The following says that we can replace the assumption of an $H_n$-center
scalar curvature lower bound by the smoothness of the tangent flow at infinity.

\begin{Corollary}\label{Coro:ancient with smooth tangent flow}
    Assume that $(M,g_t)$ is a finite-time singularity model or that $(M,g_t)$ has bounded curvature over compact time intervals with \eqref{eq: Nash entropy lower bd}.
    Suppose further that each tangent flow at infinity of $(M,g_t)$ (as defined in \cite{Bam20b}) is smooth.
    Then 
    \[
       (1+\tau) R(x,t) \ge \frac{a}{f_t(x) + C(n) - \mu_\infty},
    \]
    where $a>0$ is a constant depending on the geometry of $(M,g_t)_{t\le 0}$.

\end{Corollary}

Alternative versions of this result include that it holds under the type-I assumption or under another more technical assumption stated below.
In particular, we have the following.

\begin{Corollary}\label{Coro:ancient type-I or harnack}
    Let $(M^n,g_t)_{t\le 0}$ be a complete and non-flat ancient flow with \eqref{eq: Nash entropy lower bd}. Suppose that either:
    \begin{enumerate}
        \item $(M,g_t)$ is of type-I, i.e., $\tau |{\Rm}| \le C$ uniformly over $M\times(-\infty,0]$ for some constant $C<\infty$; or
        \item $(M,g_t)$ satisfies Hamilton's trace Harnack estimate, i.e., for any smooth vector field $X$ on $M$,
        \begin{equation}\label{eq:TheTraceHarnack}
\frac{\partial R}{\partial t}-2\langle X, \nabla R\rangle+2\Ric(X,X) \ge 0,
        \end{equation}
         and $|{\Rm}|\le CR$ for some constant $C<\infty.$
\end{enumerate}
    Then 
    \[
       (1+\tau) R(x,t) \ge \frac{a}{f_t(x) + C(n) - \mu_\infty},
    \]
    where $a>0$ is a constant depending on the geometry of $(M,g_t)_{t\le 0},$ and $\mu_\infty$ is as in \eqref{eq: Nash entropy lower bd}.

\end{Corollary}


We then apply the results above to steady gradient Ricci solitons. Recall that a triple $(M^n,g,f)$ is called a steady gradient Ricci soliton, if $(M^n,g)$ is a Riemannian manifold and $f$ is a smooth function on $M$ satisfying 
\[
\Ric=\nabla^2f.
\]
If we denote by $(\Phi_t)_{t\in\IR}$ the $1$-parameter group of diffeomorphisms generated by $-\nabla f$ with $\Phi_0={\rm id}$, then $g_t:=\Phi_t^*g$ solves the Ricci flow, and we call $(M^n,g_t)_{t\in\IR}$ the canonical form of $(M^n,g,f).$ It was proved by Hamilton \cite{Ham93} that on any steady gradient Ricci soliton, $\nabla \left(|\nabla f|^2+R\right)\equiv 0$ holds, and hence
\begin{equation}\label{bdd for naf}
    |\nabla f|^2+R=C_1,
\end{equation}
for some nonnegative constant $C_1$.
Let $|xy|:=d(x,y)$ denote the distance between $x$ and $y$; if we add a subscript $t$, then this means that the distance is with respect to $g_t$.

\begin{Definition}\label{def_noncollapsing}
A steady gradient Ricci soliton $(M^n,g,f)$ is called \emph{noncollapsed} if its canonical form $(M^n,g_t)_{t\in\mathbb{R}}$ has pointed Nash entropy bounded from below, i.e., 
\eqref{eq: Nash entropy lower bd} holds for some $o\in M$.
\end{Definition}

It is to be remarked that our definition of noncollapsing is different from that of Perelman \cite{Per02}. 
Bamler
\cite[Theorem 6.1]{Bam20a} and \cite[Proposition 4.6]{MZ21} show that Definition \ref{def_noncollapsing} implies Perelman's $\kappa$-noncollapsing on all scales. Nevertheless, Definition \ref{def_noncollapsing} is natural for applications, since due to the entropy monotonicity, any finite-time singularity model is 
necessarily
noncollapsed in the sense of Definition \ref{def_noncollapsing}.

In the case of noncollapsed steady solitons with nonnegative Ricci curvature, we can use the tangent flows at infinity to obtain curvature bounds near the $H_n$-centers, and then our estimate provides a global scalar curvature lower bound. Thus, we have the following.

\begin{Corollary}
\label{Coro: dist to H-center}
    Let $(M^n,g,f)$ be a complete and noncollapsed steady gradient Ricci soliton.
    Suppose that $\Ric\ge 0$ and $|{\Rm}|\le CR$ for some constant $C<\infty$.
    If $(M^n,g)$ is not Ricci flat,
    then for any $x\in M,\tau\ge \underline{\tau}(n),$ 
    we have that
    \[
        R(x) \ge \frac{a}{(C-\mu_\infty) \tau + |xx_\tau|^2},
    \]
    where 
    $x_\tau\in \partial B_{\tau}(o)$
    satisfies $\lambda(x_\tau,\tau_0)\le C, $ for some $\tau_0\in [\tau/C,C\tau].$ (See the definition of $\lambda$ in Section 3.3 or \cite[Section 2.1]{BCMZ21}.)
    Here, $C=C(n),$ and $a>0$ depends on the  geometry of $g.$
\end{Corollary}


We remark that this lower bound is not sharp for 
the hypothesized class of
steady solitons,
e.g., Lai's noncollapsed examples in dimension $4$ \cite{Lai20}.
However, this estimate is sharp for $3$-cylindrical steady solitons. 
In particular, if $f$ grows linearly and the level sets $\Sigma_\tau=\{f=\tau\}$ have diameters bounded by $C\sqrt{\tau}$ (this is the case when 
$R(x)\le \frac{C}{1+|xo|}$;
see Deng and Zhu \cite[Proposition 3.3]{DZ20b}), then Corollary \ref{Coro: dist to H-center} gives the improved estimate
$$R(x)\ge c/\tau\ge 
c'/|xo|,
$$
for any $x\in \Sigma_{\tau}$, if $\tau\ge \underline{\tau}(n).$
This can be viewed as 
an alternative proof of a result by Deng and Zhu
\cite[Theorem 2.5]{DZ20a} in this case.
Note that we do not need to assume the existence of critical points of $f.$

In fact, the arguments in \cite[Theorem 2.5]{DZ20a} or \cite[Proposition 4.3]{DZ18} prove the following. 
\begin{Theorem}[Deng and Zhu]
    Let $(M^n,g,f)$ be a complete noncollapsed steady gradient soliton. Suppose that $\Ric> 0$ and $|{\Rm}|\le CR$ for some constant $C<\infty.$ If $f$ admits a critical point, then 
    \[
        R(x)\ge c/|xo|,
    \]
    for any $x\notin B_{r_0}(o)$, for some constant $c>0$ and $r_0<\infty.$
\end{Theorem}
\begin{proof}[Sketch of the proof]
    \cite[Theorem 2.5]{DZ20a} uses the same arguments as \cite[Proposition 4.3]{DZ18}, where they assumed nonnegative curvature operator or bisectional curvature only to obtain Hamilton's trace Harnack estimate \eqref{eq:TheTraceHarnack} which implies a Harnack estimate for the scalar curvature \cite[(3.3)]{DZ18}. 
    As observed in \cite[Lemma 5.1]{MZ21}, nonnegative Ricci curvature on steady solitons implies Hamilton's trace Harnack estimate. The rest of the arguments are the same as in the proof of \cite[Proposition 4.3]{DZ18}.
\end{proof}

The estimate in the Corollary \ref{Coro: dist to H-center} implies that 
$$\limsup_{x\to\infty}\,R(x)|xo|>0,$$ 
which is a statement stronger than infinite ASCR. Indeed, by taking $x=x_{\tau}\in \partial B_{\tau}(o)$, 
we see that for all $\tau\gg 1$,
\[
 R(x_{\tau}) \ge \frac{a}{(C-\mu_\infty) \tau }= \frac{a}{(C-\mu_\infty)|x_\tau o|}.
\]
Hence, $\limsup_{x\to\infty}\, R(x)|xo|>0$.\smallskip

A main consequence of Theorem \ref{thm: general thm}, applicable to $4$-dimensional steady soliton singularity models, is:

\begin{Corollary}\label{Coro: 4d steady quadratic}
Let $(M^4,g,f)$ be a $4$-dimensional noncollapsed and non-Ricci-flat steady gradient Ricci soliton with bounded curvature. Let $o\in M$ be a fixed point. Then there exists a 
positive
constant $c$ depending on the soliton and $o$ such that
\begin{equation*}
    R(x)\ge\frac{c}{1+|xo|^2}.
\end{equation*}
\end{Corollary}

Hence, any $4$-dimensional steady soliton singularity model either is $\mathbb{R}^4/\Gamma$, or is a Ricci flat ALE, or has scalar curvature decaying at most quadratically.

By Munteanu, Sung, and Wang \cite{MSW19} (generalizing \cite{CLY11}), the general exponential lower bound for steady solitons $(M^n,g,f)$ with potential functions bounded from 
below
is:
\[
    R\ge ce^{-f}.
\]
Another  
asymptotic curvature estimate
for steady solitons satisfying the conditions in Corollary \ref{Coro: dist to H-center} is due to Ma and Zhang \cite[Corollary 5.2]{MZ21}:
\[
    {\rm ASCR}(g) := \limsup_{x\to\infty} R(x)
    |xo|^2 = \infty.
\]

We mention a related result by Han in \cite{Han20}, where he proves the following asymptotic curvature estimates: Let $(M^n,g,f)$ be a complete steady gradient Ricci soliton with $\sec\ge 0,$ $\Ric>0$, and where the scalar curvature decays uniformly. Then
\[
    \liminf_{x\to \infty} R(x)|xo|^{\alpha} = 0,
\]
for any $\alpha\in (0,1).$ 

We remark that one may conjecture that $R$ has an inverse linear lower bound in distance
for general steady gradient Ricci solitons without nonnegative curvature conditions.

\section{Proofs of the results}

\subsection{Proof of Theorem \ref{thm: general thm}}

\begin{proof}[Proof of Theorem \ref{thm: general thm}]

By \cite[Theorem 7.2]{Bam20a}, for any $x\in M,$ $t<0,$ we have the Gaussian upper bound for the conjugate heat kernel
\[
    |t|^{n/2}K(o,0\,|\,x,t)
    \le C\exp\left(-\mu_\infty - \frac{|xz_t|_t^2}{9|t|}\right),
\]
where $(z_t,t)$ is an $H_n$-center of $(o,0)$ and $\mu_\infty$ is the constant in \eqref{eq: Nash entropy lower bd}. Thus,
\[
    f_t(x) \geq   - C_0 +\mu_\infty+ 1+\frac{|xz_t|_t^2}{9\tau},
\]
for some $C_0=C_0(n)>0.$ In the following, we write
\begin{equation}\label{eq: rho quadratic}
        \rho_t(x) := \sqrt{1+\tau}\left(f_t(x) + C_0 - \mu_\infty\right) 
    \ge \sqrt{1+\tau} \left(1 + \frac{|xz_t|_t^2}{9\tau}\right).
\end{equation}

Since $\partial_t f = - \Delta f +|\nabla f|^2  - R  + \frac{n}{2\tau} $,
we have that 
\[
    \tau \Box f = f-\tfrac{n}{2}-w,
\]
where $\Box=\partial_t - \Delta_{g_t} $ is the heat operator. By Perelman's differential Harnack estimate \cite[\S 9]{Per02}, we have
\[
    w := \tau(2\Delta f-|\nabla f|^2+R) + f-n \le 0.
\]
So
\begin{align*}
    \tau \Box \rho &
    \ge -\frac{\tau}{2(1+\tau)}\rho + \sqrt{1+\tau} \, \big(f-\tfrac{n}{2}\big) 
    \ge  \tfrac{1}{2}\rho - A \sqrt{1+\tau},
\end{align*}
where
\[
    A=A(n,\mu_\infty) := C_0 - \mu_\infty + \tfrac{n}{2}.
\]
Then
\begin{align*}
    \tau \Box \rho^{-1} 
    &= - \tau\rho^{-2}\,\Box \rho - 2\tau\rho^{-3} \, |\nabla f|^2 \medskip  \\
    &\le -\tfrac{1}{2}\rho^{-1} + A \sqrt{1+\tau}\,\rho^{-2}.
\end{align*}
Similarly,
\begin{align*}
    \tau \Box \rho^{-2}
    &
    \le -\rho^{-2} + 2 A \sqrt{1+\tau}\,\rho^{-3},
\end{align*}
and
\[
    \tau \Box(\sqrt{1+\tau}\,\rho^{-2})
    \le - \sqrt{1+\tau}\,\rho^{-2} + 2A(1+\tau)\rho^{-3}.
\]
So
\[
    \tau \Box \big( \rho^{-1}+4A\sqrt{1+\tau}\, \rho^{-2} \big)
    \le -\tfrac{1}{2} \left(\rho^{-1}+4A\sqrt{1+\tau}\,\rho^{-2}\right)
     - A \sqrt{1+\tau}\,\rho^{-3}\left(\rho-8A\sqrt{1+\tau}\,\right).
\]
Recall that
\[
    \Box (\sqrt{1+\tau} \, R) = - \frac{R}{2\sqrt{1+\tau}} + 2\sqrt{1+\tau} \,  |{\Ric}|^2
    \ge - \frac{R}{2\sqrt{1+\tau}}.
\]
Let us define 
\begin{equation}\label{def: F}
      F := \sqrt{1+\tau} \, R - c \rho^{-1} - 4cA\sqrt{1+\tau}\,\rho^{-2},
\end{equation}
where $c=c(n,\mu_\infty,a_0)$ will be determined in the course of the proof, and where $\mu_\infty$ and $a_0$ are the constants in \eqref{eq: Nash entropy lower bd} and \eqref{eq: sc lower bd near H-centers}, respectively.
Then
\begin{equation}\label{eq: the PDE}
        \tau\Box F \ge -\tfrac{1}{2}F + cA \sqrt{1+\tau}\,\rho^{-3}\left(\rho-8A\sqrt{1+\tau}\,\right).
\end{equation}

Next, we aim to show that $\inf_{M\times(-\infty,0)} F\ge 0.$ To apply the maximum principle to \eqref{eq: the PDE}, we need to verify the boundary conditions. Note that since the operator $\tau\Box$ and the function $\rho$ become singular as $t\to 0-$, we also check the boundary condition on $M\times\{0\}$ to avoid invoking 
advanced maximum principles.\medskip  

\noindent\textbf{Claim.} \emph{By taking $c\le\overline{c}(n,\mu_\infty,a_0):=\frac{a_0}{1+4A}$, we have}
\begin{equation}\label{eq:boundary condition}
    \liminf_{x\to\infty}F(x,t)\ge 0,
\end{equation} 
\emph{uniformly for $t\in I$, where $I\subset(-\infty,0)$ is any compact interval, and}
\begin{equation}\label{eq:initial condition 1}
    \liminf_{t\to 0-}\left(\inf_{x\in M}F(x,t)\right)\ge 0,
\end{equation} 
\begin{equation}\label{eq:initial condition}
    \liminf_{t\to-\infty}\left(\inf_{x\in M}F(x,t)\right)\ge 0.
\end{equation}

\begin{proof}[Proof of the claim.]

Inequality
\eqref{eq:boundary condition} is a direct consequence of Chen's result that $R\ge 0$ (see \cite[Corollary 2.5]{Che09}), the quadratic growth of $\rho$ from \eqref{eq: rho quadratic}, the equivalence of the metrics on a compact time interval, and the fact that the $H_n$-centers of a fixed point do not drift with infinite velocity (a consequence of the estimate of Perelman's $\ell$-distance).

To prove \eqref{eq:initial condition 1}, we fix an arbitrary time $t\in(-\infty,0)$. Let $x\in M$ be an arbitrary point. If $x\in B_t(z_t,D\sqrt{1+\tau}\,)$, then by \eqref{eq: sc lower bd near H-centers} and \eqref{eq: rho quadratic}, we have
\begin{align*}
    F(x,t)\ge \frac{a_0}{\sqrt{1+\tau}}-\frac{c(1+4A)}{\sqrt{1+\tau}}\ge0,
\end{align*}
provided we take $c\le \frac{a_0}{1+4A}$. If $x\not\in B_t(z_t,D\sqrt{1+\tau})$, then we have
$$\rho(x,t)\ge \sqrt{1+\tau}\left(1+\frac{D^2(1+\tau)}{\tau}\right)\ge \frac{D^2}{\tau},$$
and thus
\begin{align*}
    F(x,t)\ge -\frac{c\tau}{D^2}-\frac{4cA\tau^2\sqrt{1+\tau}}{D^4}.
\end{align*}
Therefore, we have
$$\inf_{x\in M}F(x,t)\ge -C(A,D)\tau\to 0 \quad\text{ as }\quad t\to 0-,$$
and \eqref{eq:initial condition 1} follows immediately.

Finally, for \eqref{eq:initial condition}, recall that $R\ge 0$ and $\rho\ge \sqrt{1+\tau}$. We have
$$\inf_{x\in M}F(x,t)\ge-\frac{c(1+4A)}{\sqrt{1+\tau}}\quad\text{ for all }\quad t<0,$$
and \eqref{eq:initial condition} is proved.
\end{proof}

Assume, for a contradiction, that $\inf_{M\times(-\infty,0)}F<0$. By \eqref{eq:boundary condition}, \eqref{eq:initial condition 1}, and \eqref{eq:initial condition}, we have that the space-time infimum of $F$ must be attained at some point $(x_0,t_0)\in M\times(-\infty,0)$. Applying \eqref{eq: the PDE} at $(x_0,t_0)$, we have 
$$0\ge\tau\Box F\ge -\tfrac{1}{2}F + cA \sqrt{1+\tau}\,\rho^{-3}\left(\rho-8A\sqrt{1+\tau}\,\right)\quad\text{at}\quad (x_0,t_0).$$
Since $F(x_0,t_0)<0$, we have
$$8A\sqrt{1+|t_0|}\ge \rho(x_0,t_0)\ge \sqrt{1+|t_0|} \left(1 + \frac{|x_0z_{t_0}|_{t_0}^2}{9|t_0|}\right),$$
where we have applied \eqref{eq: rho quadratic}. Thus, we have
\begin{equation}\label{eq: dist estimate}
    |x_0z_{t_0}|_{t_0}\le\sqrt{72A(1+|t_0|)}.
\end{equation}

If we require the constant $D$ in \eqref{eq: sc lower bd near H-centers} to satisfy $D\ge \sqrt{100A}$, then by \eqref{eq: sc lower bd near H-centers}, we have 
$$R(x_0,t_0)\ge \frac{a_0}{1+|t_0|}.$$
By the definitions of $F$ in \eqref{def: F} and of $\rho$ in \eqref{eq: rho quadratic}, we have
\begin{align*}
    F(x_0,t_0)\ge \frac{a_0}{\sqrt{1+|t_0|}}-\frac{c(1+4A)}{\sqrt{1+|t_0|}}\ge 0,
\end{align*}
provided we take $c\le \frac{a_0}{1+4A}$; this is a contradiction, and the proof of Theorem \ref{thm: general thm} is complete.
\end{proof}

\subsection{Proofs of Corollaries \ref{Coro:ancient with smooth tangent flow}
and \ref{Coro:ancient type-I or harnack}}

\begin{proof}[Proof of Corollary \ref{Coro:ancient with smooth tangent flow}]
First of all, the ancient solution in question cannot be Ricci-flat. Since the tangent flow at infinity of a Ricci-flat static Ricci flow is its 
own
blow-down limit, which necessarily contains a singularity unless the manifold is flat.

We show by contradiction that \eqref{eq: sc lower bd near H-centers} holds for any $D>0$. Suppose that \eqref{eq: sc lower bd near H-centers} is false for some $D>0$. Since non-Ricci-flat ancient Ricci flows have positive scalar curvature \cite{Che09}, there must be a sequence $(x_i,t_i)$ with $t_i\to-\infty$ and $x_i\in B_{t_i}(z_{t_i},D\sqrt{1+|t_i|})$, such that $(1+|t_i|)R(x_i,t_i)\to0$. 
Take $|t_i|$ as the scaling factors when obtaining the tangent flow and let $z_{t_i}\to z_\infty\in M_{\infty}$. Then there is a point $x_\infty\in B(z_\infty, D)$ such that $R_\infty(x_\infty)=0$. By the strong maximum principle, we immediately have that $(M_\infty,g_\infty,f_\infty)$ is the Euclidean space, and $\mu_\infty=0$. The ancient solution in question must also be Euclidean space by Perelman's monotonicity formula; this is a contradiction.
\end{proof}

\begin{proof}[Proof of Corollary \ref{Coro:ancient type-I or harnack}]
    By \cite[Proposition 2.4]{MZ21}, under the assumptions of Hamilton's trace Harnack estimate \eqref{eq:TheTraceHarnack}
    and the curvature pinching condition $|{\Rm}|\le CR,$ Perelman's asymptotic shrinkers exist and are smooth. Thus, by \cite{CMZ21a}, each tangent flow at infinity is smooth. We may now apply Corollary \ref{Coro:ancient with smooth tangent flow}.

    For the type-I case, Cao and Zhang proved the existence of Perelman's asymptotic shrinkers in \cite[Theorem 4.1]{CZ11} and thus we can apply Corollary \ref{Coro:ancient with smooth tangent flow}.
\end{proof}

\subsection{Proofs of Corollaries  \ref{Coro: dist to H-center} and \ref{Coro: 4d steady quadratic}}

Corollary \ref{Coro: dist to H-center} bounds the scalar curvature from below according to the distance to certain anchor points on the level sets which play the same role as $H_n$-centers or $\ell$-centers (points where the $\ell$-distance is bounded by a fixed constant). 

For steady solitons, it 
is often
more convenient to work with the static metric as compared to the induced Ricci flow. Let us recall some notions from \cite{BCMZ21}.
Let $(M^n,g,f)$ be a steady gradient Ricci soliton and let $\Phi_t$ be the $1$-parameter group of diffeomorphisms generated by $-\nabla f.$
Fix a point $o\in M.$
We define
\[
\Lambda(x,\tau)
    := \inf \int_0^{\tau} \sqrt{s} \, \big(R_g+|\dot\gamma -\nabla f|_g^2 \big) (\gamma(s))\, ds,
\]
where the infimum is taken over all $\gamma:[0,\tau]\to M$ with $\gamma(0)=o$ and $\gamma(\tau)=x$.
Accordingly, define
\[
    \lambda(x,\tau) := \ell(\Phi_\tau(x),\tau)
    =: \frac{1}{2\sqrt{\tau}}\Lambda(x,\tau),
\]
where $\ell$ is Perelman's $\ell$-function.
Arguing as Perelman in \cite[Section 7.1]{Per02}, we have that, for any $\tau>0,$ there is a point $p_\tau\in M$ such that
$$\lambda(p_\tau,\tau)=\ell(\Phi_{\tau}(p_\tau),\tau) \le \frac{n}{2}.$$
Any such point $p_\tau$ is called a \textbf{$\lambda$-center} at time $-\tau$. See \cite[Section 2.1]{BCMZ21} for more details.


\begin{proof}[Proof of Corollary \ref{Coro: dist to H-center}] 
Recall that by Perelman's space-time comparison geometry, the trace Harnack estimate \eqref{eq:TheTraceHarnack} implies that 
\begin{equation}
\label{ineq: grad ell le C ell}
    \tau |\nabla \ell|^2 \le C\ell.
\end{equation}
By \cite[Lemma 2.4]{BCMZ21}, for $\tau\ge \bar\tau(n),$ we can find a point $x_\tau\in \partial B_{\tau}(o)$ satisfying
    \[
\lambda(x_\tau,\tau_0)\le C 
    \]
    for some $\tau_0\in [\tau/C,C\tau]$, where $C=C(n).$
    We write $\bar x_{\tau}=\Phi_{\tau_0}(x_{\tau}).$ Then
    \[
        \ell(\bar x_{\tau},\tau_0)
        = \lambda(x_\tau,\tau_0)\le C.
    \]
    By \cite[Corollary 9.5]{Per02} and \eqref{ineq: grad ell le C ell}, for any $\bar x=\Phi_{\tau_0}(x)\in M,$
    \[
        f_{-\tau_0}(\bar x)
        \le \ell(\bar x,\tau_0)
        \le C +C\frac{|\bar x\bar x_{\tau}|_{-\tau_0}^2}{\tau_0}
        \le C +C\frac{|xx_{\tau}|^2}{\tau},
    \]
    where $(4\pi |t|)^{-\frac{n}{2}} e^{-f_t(x)}=K(o,0\,|\,x,t).$

    By \cite[Theorem 1.13]{MZ21}, Perelman's asymptotic shrinkers exist and are smooth. We may now apply Corollary \ref{Coro:ancient with smooth tangent flow} to conclude that, for any $\bar x=\Phi_{\tau_0}(x)\in M,$ $\tau\ge \underline{\tau}(n),$
    \begin{align*}
        R(x)&= R(\Phi_{-\tau_0}(\bar x))
        = R(\bar x,-\tau_0) \\
        & \ge \frac{a}{\tau_0(f_{-\tau_0}(\bar x)+C(n)-\mu_\infty)}\\
        &\ge  \frac{a'}{(C(n)-\mu_\infty) \tau + |xx_\tau|^2}.
    \end{align*}
\end{proof}

\begin{proof}[Proof of Corollary \ref{Coro: 4d steady quadratic}]
Let $(M,g_t)_{t\in(-\infty,\infty)}$ be the canonical form of the steady soliton in question. Since the curvature and the pointed Nash entropy are both uniformly bounded, Bamler's theorems \cite{Bam20a,Bam20b,Bam20c} can be applied, due to Bamler's explanation in \cite{Bam21}. Thus, we may apply \cite[Proposition 3.1]{BCDMZ21} to the ancient solution $(M,g_t)_{t\in(-\infty,0]}$. By the classification of tangent flows at infinity therein and the non-Ricci-flat assumption, we have that the tangent flow at infinity is unique and smooth. Thus, we may apply Corollary \ref{Coro:ancient with smooth tangent flow}. Finally, fixing, say, $\tau=1$, and applying a standard Gaussian lower bound to the conjugate heat kernel (see e.g.~\cite[Theorem 26.31]{RFV3}), we obtain the quadratic decay lower bound for the scalar curvature. More precisely, by 
the 
Bishop--Gromov volume comparison theorem, 
the volumes of unit balls are uniformly bounded from above: for any $x\in M$,
\[
|B_1(x)|_g\leq C,
\]
where $|B_1(x)|_g$ is the volume of the unit geodesic ball taken with respect to the soliton metric $g$ and $C$ depends on the curvature bound of $g$. It can be seen from \cite[Theorem 26.31]{RFV3} that
\[
(4\pi )^{-\frac{n}{2}} e^{-f_{-1}(x)}=K(o,0\,|\,x,-1)\ge \frac{c_1}{\sqrt{|B_1(o)|_g}\sqrt{|B_1(x)|_g}}e^{-\frac{|ox|^2}{c_2}} \geq \frac{c_1}{C}e^{-\frac{|ox|^2}{c_2}} ,
\]
where $c_1$ and $c_2$ are positive constants depending only on the curvature bound of the soliton. Hence, by $|\nabla f|\leq \sqrt{C_1}$ \eqref{bdd for naf}, for any $x\in M$,
$$f_{-1}(\Phi_{1}(x))\leq C(|o\Phi_{1}(x)|^2+1)\leq  2C(|x\Phi_{1}(x)|^2+|ox|^2+1)\leq 2C(|ox|^2+C_1+1).$$ 
The quadratic decay bound of the scalar curvature now follows by applying Corollary \ref{Coro:ancient with smooth tangent flow} to the canonical form $(M,g_t)_{t\in(-\infty, 0]}$.
\end{proof}

\subsection{Recovering the shrinker estimate for singularity models}

In this section we show that we can consider Theorem \ref{thm: general thm} as a generalization of the lower scalar curvature estimate for shrinkers of \cite{CLY11}, at least in the case of singularity models.

Let $(M^n,g,\bar f)$ be a complete shrinking gradient Ricci soliton satisfying
    \[
        \Ric+\nabla^2\bar f = \tfrac{1}{2}g,\quad
        R+|\nabla \bar f|^2 = \bar f.
    \]
Let $(\Phi_t)_{t\in \IR}$ be the $1$-parameter group of diffeomorphisms generated by $\nabla \bar f.$ Then
\[
    g_t := |t| \Phi^*_{-\ln|t|} g
\]
solves the Ricci flow and is called the canonical form generated by the shrinking soliton. 
Recall that if we set 
$$\bar f_t := \Phi_{-\ln|t|}^*\bar f,$$
then
\[
\Ric_{g_t} + \nabla^2_{g_t}\bar f_t
= \tfrac{1}{2\tau} g_t,\quad
\tau(R_{g_t}+|\nabla \bar f_t|^2_{g_t})=\bar f_t.
\]
We have the following simple lemma.
\begin{Lemma}
    Let $(M^n,g_t)_{t<0}$ be the canonical form induced by a complete shrinking gradient Ricci soliton $(M^n,g,\bar f)$. Then for any $x\in M,$ $t<-1,$
    \[
        \int_M |t|\bar f_t \, d\nu_{x,-1|t}
        = \bar f(x) + \tfrac{n}{2}(-1-t).
    \]
\end{Lemma}
\begin{proof}
For any $t<-1,$ writing $\nu_t=\nu_{x,-1|t},$
    \begin{align*}
        \partial_t \int_M \bar f_t \, d\nu_t
        &= \int_M \Box \bar f_t \, d\nu_t
        = \frac{1}{\tau}\int_M \bar f_t\, d\nu_t- \frac{n}{2\tau}.
    \end{align*}
    See, e.g., \cite{LW20} for the justification of integration by parts at infinity.
    Integrating the equation above from $t$ to $-1$, we have
    \[
        \bar f(x) - \tau\int_M \bar f_t \, d\nu_t 
        = -\tfrac{n}{2}(-1-t),
    \]
    and the conclusion follows.
\end{proof}

As a corollary, 
the
$H_n$-centers of any minimum point $o$ of the potential function $\bar f$ on shrinkers are not far away from $o$.
\begin{Proposition}
\label{prop: H-centers on shrinkers}
    Let $(M^n,g_t)_{t<0}$ be the canonical form induced by a complete shrinking gradient Ricci soliton $(M^n,g,\bar f)$. 
    Let $o$ be a minimum point of $\bar f$. Then for any $t<-1,$ if $(z,t)$ is an $H_n$-center of $(o,-1),$ then
    \[
        |oz|_t \le C(n)\sqrt{\tau}.
    \]
\end{Proposition}
\begin{proof}
    By the fundamental growth estimate of the potential function $\bar f$ by Cao and Zhou \cite{CZ10}, 
    for any $y\in M,t<-1,$
    \[
        \tau\bar f_t(y)
        \ge \tfrac{1}{4}\left(|oy|_t -C\sqrt{\tau}\right)_+^2,
    \]
    where $C=C(n).$
    Write $\nu_t=\nu_{o,-1|t}.$
    Then for any $t<-1,$
    \begin{align*}
        \int_M |oy|_t^2\, d\nu_{t}(y) 
        &\le   \int_M \left(|oy|_t-C\sqrt{\tau}\right)_+^2 
        \, d\nu_t(y)
        + C\tau\\
        &\le  C\tau + 4\tau\int_M \bar f_t \, d\nu_t\\
        &\le C\tau,
    \end{align*}
    where
    we have used the previous lemma and the fact that $\bar f(o)\leq n/2$ in the last inequality.
Thus, by \cite[Lemma 3.2]{Bam20a},
\begin{align*}
    |oz|_t &\le
    \dist_{W_1}^{g_t}(\delta_{o}, \nu_t)
    +  \dist_{W_1}^{g_t}(\delta_{z}, \nu_t)
    \le  C\sqrt{\tau}.
\end{align*}
\end{proof}

\begin{Corollary}
    Let $(M^n,g_t)_{t<0}$ be the canonical form induced by a complete non-flat shrinking gradient Ricci soliton $(M^n,g,\bar f)$. 
    Suppose also that it is a singularity model or $(M^n,g)$ has bounded curvature. Then for any $D<\infty,$ the condition \eqref{eq: sc lower bd near H-centers} holds for some $a_0>0$, after shifting the time by $-1$.
    As a consequence of Theorem \ref{thm: general thm}, we can recover the main theorem of \cite{CLY11} in this case.
\end{Corollary}
\begin{proof}
    Let $o$ be a minimum point of $\bar f.$ By Proposition \ref{prop: H-centers on shrinkers}, for any $t<-1$, and any $H_n$-center $(z_t,t)$ of $(o,-1),$
    \[
        |oz_t|_t \le C\sqrt{\tau}.
    \]
    Since $(M^n,g)$ is not Ricci-flat,  we can choose $a_0>0$ such that
    \[
        \inf_{B(o,2(D+C))} R \ge a_0.
    \]
    Note that
    \[
        B_t\big(z_t,D\sqrt{1+\tau}\,\big)
        \subset B_t\big(o,(D+C)\sqrt{1+\tau}\,\big)
        = \Phi_{\ln|t|}\left(B\left(o, (D+C)\sqrt{(1+\tau)/\tau} \right)\right).
    \]
    Recall that $R(x,t)=R(\Phi_{-\ln|t|}(x))/\tau.$
    Thus, \eqref{eq: sc lower bd near H-centers} holds.
\end{proof}

\bibliographystyle{amsalpha}

\newcommand{\alphalchar}[1]{$^{#1}$}
\providecommand{\bysame}{\leavevmode\hbox to3em{\hrulefill}\thinspace}
\providecommand{\MR}{\relax\ifhmode\unskip\space\fi MR }
\providecommand{\MRhref}[2]{%
  \href{http://www.ams.org/mathscinet-getitem?mr=#1}{#2}
}
\providecommand{\href}[2]{#2}

\end{document}